\documentclass[letterpaper,twoside,web]{ieeecolor}

\usepackage{generic}
\usepackage{cite}
\usepackage{amsmath,amssymb,amsfonts}
\usepackage{graphicx}
\usepackage{hyperref}
\hypersetup{hidelinks}
\usepackage{textcomp}
\usepackage{eso-pic}
\usepackage{lcsys}

\def\BibTeX{{\rm B\kern-.05em{\sc i\kern-.025em b}\kern-.08em
    T\kern-.1667em\lower.7ex\hbox{E}\kern-.125emX}}
\newtheorem{assumption}{Assumption}[section]
\newtheorem{definition}{Definition}[section]

\newtheorem{theorem}{Theorem}[section]
\newtheorem{proposition}{Proposition}[section]

\newtheorem{problem}{Problem}[section]

\newcommand{\R}{\mathbb{R}}
\newcommand{\Z}{\mathbb{Z}}
\newcommand{\RH}{\mathcal{RH}}
\newcommand{\He}{\mathcal{H}}
\newcommand{\st}{\text{s.t.}}
\newcommand{\ind}[1]{\mathbf{1}_{\{#1\}}}
\newcommand{\Tlat}{T_{\mathrm{lat}}}
\newcommand{\TFIR}{T_{\mathrm{FIR}}}
\newcommand{\Rp}{\mathcal{R}_{\mathrm p}}
\newcommand{\Aadm}{\mathcal{A}_{\mathrm{adm}}}
\newcommand{\BadmT}{\mathcal{B}_{\mathrm{adm}}^{\Tlat}}

\DeclareMathOperator{\col}{col}
\DeclareMathOperator{\Tr}{Tr}
\begin{document}
\AddToShipoutPictureFG*{%
  \AtPageLowerLeft{%
    \raisebox{0.18in}[0pt][0pt]{%
      \makebox[\paperwidth][c]{%
        \parbox{0.94\paperwidth}{\centering\scriptsize
        This work has been submitted to the IEEE for possible publication.
        Copyright may be transferred without notice, after which this version
        may no longer be accessible.}}}}}
\title{Deployment-Aware Controller and Control Architecture Co-Design via Mixed-Integer Output-Feedback SLS}
\author{Chenchen Zhou and Jose Matias
\thanks{The postdoctoral fellowship of C. Zhou was funded by KU Leuven through project ZKE5508.}
\thanks{Chenchen Zhou and Jose Matias are with KU Leuven, Chemical and Biochemical Reactor Engineering and Safety (CREaS), De Nayer Campus, Jan Pieter de Nayerlaan 5, 2860 Sint-Katelijne-Waver, Belgium (e-mail: chenchen.zhou@kuleuven.be; jose.matias@kuleuven.be).}
\thanks{OpenAI's ChatGPT was used only for language polishing and spelling and grammar checks.}}

\maketitle
\thispagestyle{empty}

\begin{abstract}
We study controller and control-architecture co-design for output-feedback
systems under a hard budget. The architecture activates sensors and actuators
and selects among directed communication-service options with specified
delivery bounds and costs. Direct optimization over controller transfer
matrices and discrete deployments is mixed-integer nonconvex; convex
alternatives fix the architecture, use regularization, or impose a
quadratically invariant (QI) controller-information pattern. We instead
optimize finite impulse response (FIR) output-feedback system-level synthesis
(OF-SLS) responses. Binary variables select devices and service options;
cumulative binaries record whether the selected service can deliver by each FIR
lag. Indicator constraints zero response coefficients that would require
unavailable messages. For fixed device and OF-SLS realization-state locations,
this yields an exact mixed-integer convex program (MICP) over finite service
menus and deployment constraints. Every feasible response admits the standard
OF-SLS implementation using only the selected devices and services. In a
three-follower platoon, 2736 of 139,968 stabilizable and detectable deployments
are QI-compatible. All three actuators are necessary, whereas
intermediate-budget optima retain strict subsets of seven sensor packages. At a
common budget, the best QI design has 3.86 times the performance loss of the
co-design optimum relative to the dense deployment.
\end{abstract}


\section{Introduction}
\label{sec:introduction}
\IEEEPARstart{C}{yber-physical} control systems couple physical dynamics with
sensing, actuation, communication, computation, and software modules. Their
architecture is therefore part of the control design: the designer must decide
which sensors and actuators are installed and which sensor measurements reach
each actuator-side controller component at which delay. These decisions are
especially visible in networked systems such as vehicle platoons, power
networks, and multi-agent systems. The architecture determines which
disturbance information can reach which actuator before the disturbance
propagates through the physical interconnection. Missing or slow links force
the controller to react through delayed information paths, so even the best
controller for the fixed architecture can incur a larger closed-loop cost.

Existing tools cover fixed or convexifiable architectures. Quadratic invariance
(QI) requires the plant-induced feedback transformation to preserve a
controller-information subspace and yields a convex Youla formulation;
sparsity invariance gives related convex restrictions
\cite{rotkowitz2005,furieriSparsityInvariance2020}. System-level synthesis
(SLS), Youla, and
input--output parameterizations equivalently describe stabilizing controllers;
SLS uses achievable closed-loop responses
\cite{wangSystemLevelApproach2019,zhengEquivalenceYoulaSystemLevel2021}.
For a fixed architecture, affine output-feedback SLS (OF-SLS) constraints
retain convexity and may induce non-QI controller support. Localized synthesis
and decay enable distributed design \cite{wang2018b,shinNearOptimalDistributed2023},
whereas deployment analysis maps controllers to realization resources
\cite{tsengAndersonSynthesisDeployment2020}. These concern distinct controller
support, response support, realization, and deployment structures.

With these tools available, the remaining question is how to choose hardware,
information links, and delays under cost and design rules. Related work treats
structural and joint sensor--control-node selection
\cite{pequitoStructuralConfiguration2016,nugrohoJointSensorControl2019},
communication--compute--control interactions
\cite{grosjeanFrameworkCommunicationComputeControl2025}, and sparse architectures
through regularization for design (RFD)
\cite{matniRegularizationDesign2016,matniCommunicationDelayCoDesign2017}.
Recent work uses evolutionary pruning for actuator--sensor--communication
co-design \cite{wuEvolutionaryActuatorSensorCommunication2026}, but does not
couple finite delay-rated end-to-end service menus to a
deployment-certified OF-SLS realization under a hard budget.
Here, exact budgets, finite delay-rated services, compatibility rules, and a
realization certificate are imposed together.

To close this gap, the paper couples finite hardware and directed-service
decisions to finite impulse response (FIR) OF-SLS responses. Fixed candidate
locations, binary selections, cumulative delivery variables, and indicator
gates yield an exact finite mixed-integer convex program (MICP). Within the
prescribed locations, service menus, FIR horizon, and deployment constraints,
a zero-gap solve certifies the best architecture-response pair; each feasible
response has a standard OF-SLS implementation using selected resources. A
vehicle-platoon study compares these optima with uniformly re-synthesized QI,
standard-pattern, and RFD architectures under common hard budgets.

The remainder formulates the problem (Section~\ref{sec:problem_formulation}),
derives the MICP and QI reference (Section~\ref{sec:binary_sls_reformulation}),
presents the numerical study (Section~\ref{sec:numerics}), and concludes
(Section~\ref{sec:discussion}).

\section*{Notation}
For integers \(a\le b\), let
\(\Z_{[a,b]}:=\{a,a+1,\ldots,b\}\) and
\(\Z_{\ge0}:=\{0,1,\ldots\}\). For
\(F(z)=\sum_{t\ge0}F[t]z^{-t}\), \(F[t]\) denotes its \(t\)-th impulse
coefficient. The spaces of proper stable and proper real-rational transfer
matrices are denoted by \(\RH_\infty\) and \(\Rp\), respectively, and
\(\ind{\mathsf{condition}}\) equals one when the condition holds and zero
otherwise. Delay
values use the usual extended order and arithmetic on
\(\Z_{\ge0}\cup\{\infty\}\).

\section{Information-Latency Architecture Co-Design Problem}
\label{sec:problem_formulation}
\subsection{Output-Feedback SLS Parameterization and Performance}
Consider the discrete-time linear time-invariant (LTI) generalized plant
\begin{equation}
\label{eq:generalized_plant}
\begin{aligned}
x(t+1) &= A x(t) + B_1 w(t) + B_2 u(t), \\
\zeta(t)   &= C_1 x(t) + D_{11} w(t) + D_{12} u(t), \\
y(t)   &= C_2 x(t) + D_{21} w(t) + D_{22} u(t),
\end{aligned}
\end{equation}
where $x(t)\in\R^{n_x}$ is the plant state, $w(t)\in\R^{n_w}$ is the disturbance input, $u(t)\in\R^m$ is the control input, $y(t)\in\R^p$ is the measured output, and $\zeta(t)\in\R^{n_\zeta}$ is the regulated output. The control input and measured output are partitioned into actuator and sensor channels:
\begin{equation}
\label{eq:input_output_partition}
\begin{aligned}
u(t) &= \col\big(u_1(t),\dots,u_{n_u}(t)\big),\qquad
u_i(t)\in\R^{m_i}, \\
y(t) &= \col\big(y_1(t),\dots,y_{n_y}(t)\big),\qquad
y_j(t)\in\R^{p_j},
\end{aligned}
\end{equation}
where $\sum_i m_i=m$ and $\sum_j p_j=p$. The plant transfer matrix from $u$ to
$y$ is \(P_{22}(z):=D_{22}+C_2(zI-A)^{-1}B_2\).
\begin{assumption}
\label{ass:plant}
The generalized plant satisfies $D_{22}=0$, $(A,B_2)$ is stabilizable, and $(A,C_2)$ is detectable.
\end{assumption}
Under Assumption~\ref{ass:plant}, \(P_{22}\) is strictly proper. Hence the FIR
class may include the instantaneous response coefficient \(L[0]\) without an
algebraic loop or an additional well-posedness constraint.

Let \(\delta_x:=B_1w\) and \(\delta_y:=D_{21}w\) denote the disturbances
entering the state and measurement equations in \eqref{eq:generalized_plant}.
An OF-SLS system response is the closed-loop map from
\((\delta_x,\delta_y)\) to \((x,u)\):
\begin{equation}
\label{eq:Psi_def}
\begin{bmatrix}
x\\u
\end{bmatrix}
=
\Psi
\begin{bmatrix}
\delta_x\\ \delta_y
\end{bmatrix},
\qquad
\Psi:=
\begin{bmatrix}
R & N \\
M & L
\end{bmatrix},
\end{equation}
where \(R,N,M,L\) have dimensions conforming to
\eqref{eq:input_output_partition}. These blocks are closed-loop responses, not
controller-transfer blocks. We call \(\Psi\) stable and achievable if it
satisfies the stability constraints
\begin{equation}
\label{eq:response_spaces}
\begin{aligned}
R&\in z^{-1}\RH_\infty^{n_x\times n_x},&
M&\in z^{-1}\RH_\infty^{m\times n_x},\\
N&\in z^{-1}\RH_\infty^{n_x\times p},&
L&\in \RH_\infty^{m\times p}
\end{aligned}
\end{equation}
and the OF-SLS affine constraints
\begin{equation}
\label{eq:ofsls_achievability}
\begin{aligned}
\begin{bmatrix}
zI-A & -B_2
\end{bmatrix}
\begin{bmatrix}
R & N \\
M & L
\end{bmatrix}
&=
\begin{bmatrix}
I & 0
\end{bmatrix}, \\
\begin{bmatrix}
R & N \\
M & L
\end{bmatrix}
\begin{bmatrix}
zI-A \\
-C_2
\end{bmatrix}
&=
\begin{bmatrix}
I \\
0
\end{bmatrix}.
\end{aligned}
\end{equation}

Under Assumption~\ref{ass:plant}, stable response tuples \((R,N,M,L)\) satisfying
\eqref{eq:response_spaces} and both identities in
\eqref{eq:ofsls_achievability} are in one-to-one
correspondence with internally stabilizing output-feedback controllers
\cite{wangSystemLevelApproach2019,wang2018b}. The corresponding dynamic
controller is \(K=L-MR^{-1}N\), and its induced \(L\)-block satisfies
\(L=K(I-P_{22}K)^{-1}\). Support constraints on \(\Psi\) specify the admissible
closed-loop response blocks; the realization messages and deployed resources
are specified below.

For any feasible \(\Psi\), the closed-loop map from \(w\) to \(\zeta\) is
affine:
\begin{equation}
\label{eq:Gzetaw_Psi}
G_{\zeta w}(\Psi)
=
D_{11}
+
\begin{bmatrix}
C_1 & D_{12}
\end{bmatrix}
\begin{bmatrix}
R & N \\
M & L
\end{bmatrix}
\begin{bmatrix}
B_1 \\
D_{21}
\end{bmatrix}.
\end{equation}
For a chosen closed-loop performance functional \(\mathcal J\), define
\(J_{\mathrm{perf}}(\Psi):=\mathcal J(G_{\zeta w}(\Psi))\); it is convex in
\(\Psi\) when \(\mathcal J\) is convex.

Fix an FIR horizon \(\TFIR\). The FIR response class uses
\begin{equation}
\label{eq:FIR_param}
\begin{aligned}
R(z) &= \sum_{t=1}^{\TFIR} R[t] z^{-t}, \qquad
M(z) = \sum_{t=1}^{\TFIR} M[t] z^{-t}, \\
N(z) &= \sum_{t=1}^{\TFIR} N[t] z^{-t}, \qquad
L(z) = \sum_{t=0}^{\TFIR} L[t] z^{-t}.
\end{aligned}
\end{equation}

\begin{definition}[Finite-dimensional FIR response class]
\label{def:fir_class}
The class \(\mathcal F_{\TFIR}\) consists of response tuples of the form
\eqref{eq:FIR_param} satisfying the finite affine SLS equations obtained from
\eqref{eq:ofsls_achievability}. Equivalently, with
\(R[0]=M[0]=N[0]=0\) and
\(R[\TFIR+1]=M[\TFIR+1]=N[\TFIR+1]=0\), the left and right OF-SLS identities
become
\begin{equation}
\label{eq:FIR_recursions}
\begin{aligned}
R[t+1] &= A R[t] + B_2 M[t] + \ind{t=0}I,\\
N[t+1] &= A N[t] + B_2 L[t],\\
R[t+1] &= R[t]A + N[t]C_2 + \ind{t=0}I,\\
M[t+1] &= M[t]A + L[t]C_2,
\end{aligned}
\end{equation}
where $t\in\Z_{[0,\TFIR]}$.
The two \(R\)-updates are the coefficient forms of the left and right OF-SLS identities;
 \(L\) is coupled to the terminal constraints through the \(N\)- and \(M\)-updates.
\end{definition}

Within \(\mathcal F_{\TFIR}\), let \(J_{\mathrm{FIR}}\) denote the
finite-dimensional restriction of \(J_{\mathrm{perf}}\). For the squared
finite-horizon \(\He_2\) criterion,
\begin{equation}
\label{eq:J_FIR}
J_{\mathrm{FIR}}(\Psi)
:=
\sum_{t=0}^{\TFIR}
\Tr\!\big( G_{\zeta w}[t]^\top G_{\zeta w}[t] \big),
\end{equation}
where \(G_{\zeta w}[t]\) is the \(t\)-th impulse coefficient of
\(G_{\zeta w}(\Psi)\). Through the terminal equalities and finite sum,
\(\TFIR\) can affect feasibility, performance, and the selected architecture.
Under the natural zero-padding embedding,
\(\mathcal F_T\subseteq\mathcal F_{T+1}\), so the optimal value for a fixed
deployment is nonincreasing in (T). This nesting alone does not establish
convergence to the unrestricted stable-response optimum or eventual invariance
of the selected architecture.

To connect response constraints to deployable messages, we use the standard
OF-SLS realization available for every achievable \(\Psi\):
\begin{equation}
\label{eq:ofsls_realization}
\begin{aligned}
z\beta&=\widetilde R^+\beta+\widetilde N y,\\
u&=\widetilde M\beta+Ly,
\end{aligned}
\qquad
\begin{aligned}
\widetilde R^+&=z(I-zR),\\
\widetilde M&=zM,\qquad \widetilde N=-zN.
\end{aligned}
\end{equation}
Here \(\beta\in\R^{n_x}\) is the realization state and \(\widetilde R^+\) the
shifted \(R\)-filter. Eliminating \(\beta\) yields
\(K=L-MR^{-1}N\); at realization-filter lag \(q\ge0\),
\begin{equation}
\label{eq:realization_shifts}
\widetilde M[q]=M[q+1],\
\widetilde N[q]=-N[q+1],\
\widetilde R^+[q]=-R[q+2].
\end{equation}

\subsection{Hardware, Communication Services, and Site Assignments}

To assign the OF-SLS realization in \eqref{eq:ofsls_realization} to sites, let
\(\mathcal V\) be a fixed set of deployment sites and partition
\(\beta=\col(\beta_1,\ldots,\beta_{n_\beta})\), where
\(\beta_k\in\R^{n_{\beta,k}}\) and
\(\sum_{k=1}^{n_\beta}n_{\beta,k}=n_x\). The partition is conformal with the
block rows and columns of \(R\), the columns of \(M\), and the rows of \(N\). The
site-assignment maps
\begin{equation}
\label{eq:hosts}
p_u(i)\in\mathcal V,\qquad p_y(j)\in\mathcal V,
\qquad p_\beta(k)\in\mathcal V
\end{equation}
assign actuator \(u_i\), sensor \(y_j\), and realization-state block
\(\beta_k\) to sites. Multiple objects may share a site. These assignments are
specified before co-design and define the deployment class.

Independent binaries
\(\eta_i,\xi_j\in\{0,1\}\) enable actuator \(i\) and sensor \(j\), respectively.

Let \(\Tlat\in\Z_{\ge1}\) denote the maximum selectable communication delay.
For each ordered site pair \((s,r)\), the design selects
\begin{equation}
\label{eq:service_menu}
d_{rs}\in\mathcal D_{rs},\qquad
\mathcal D_{rs}\subseteq\Z_{[0,\Tlat]}\cup\{\infty\}.
\end{equation}
Finite entries need not be consecutive; \(\infty\) denotes no service. Each
finite option specifies an available lossless end-to-end service that delivers
from \(s\) to \(r\) within \(d_{rs}\) samples. A fixed physical delay is
represented by a singleton \(\mathcal D_{rs}\). Sufficient capacity prevents
queueing; early arrivals may be buffered.

Let \(d:=[d_{rs}]_{r,s\in\mathcal V}\),
\(\eta:=\col(\eta_1,\ldots,\eta_{n_u})\), and
\(\xi:=\col(\xi_1,\ldots,\xi_{n_y})\). Device activation is separate from fixed
site computation and memory, so a site may relay messages or maintain its
\(\beta\)-block without activating a colocated device. Let
\(c_i^{\mathrm a},c_j^{\mathrm s}\ge0\) be actuator and sensor
costs, and let \(c_{rs}^{\mathrm c}(\nu)\ge0\) be the one-time cost of service
option \(\nu\), with \(c_{rs}^{\mathrm c}(\infty)=0\).
A selected service may carry multiple realization messages but is charged once per directed
site pair.
Local transfers, computation, and
memory are free; for each site, \(d_{rr}=0\) and \(c_{rr}^{\mathrm c}(0)=0\).
Define
\begin{equation}
\label{eq:architecture_cost}
\begin{aligned}
J_{\mathrm{arch}}(d,\eta,\xi)
=\sum_i c_i^{\mathrm a}\eta_i+\sum_j c_j^{\mathrm s}\xi_j
+\sum_{r,s\in\mathcal V}c_{rs}^{\mathrm c}(d_{rs}).
\end{aligned}
\end{equation}
Let \(\Aadm\) collect the service-option membership and binary-domain constraints,
together with any hard cardinality,
compatibility, mandatory, or forbidden-resource constraints; hence
\((d,\eta,\xi)\in\Aadm\) is an admissible deployment.

The four message families are $y_j\to u_i$ for $L_{ij}$,
$\beta_k\to u_i$ for $\widetilde M_{ik}$,
$y_j\to\beta_k[t+1]$ for $\widetilde N_{kj}$, and
$\beta_\ell\to\beta_k[t+1]$ for $\widetilde R^+_{k\ell}$. Their respective
source--destination site pairs are
\begin{equation}
\label{eq:realization_message_sites}
\begin{split}
p_y(j)&\to p_u(i),\qquad p_\beta(k)\to p_u(i),\\
p_y(j)&\to p_\beta(k),\qquad p_\beta(\ell)\to p_\beta(k).
\end{split}
\end{equation}

The \(L[q]\) and \(\widetilde M[q]\) terms form \(u[t]\), whereas
\(\widetilde N[q]\) and \(\widetilde R^+[q]\) form \(\beta[t+1]\). The shifts
in \eqref{eq:realization_shifts} therefore translate the message deadlines into
the response-coefficient constraints
\begin{equation}
\label{eq:implementation_mask_rule}
\begin{aligned}
L_{ij}[t]&=0&&\text{if }d_{p_u(i),p_y(j)}>t,\\
M_{ik}[t]&=0&&\text{if }d_{p_u(i),p_\beta(k)}>t-1,\\
N_{kj}[t]&=0&&\text{if }d_{p_\beta(k),p_y(j)}>t,\\
R_{k\ell}[t]&=0&&\text{if }d_{p_\beta(k),p_\beta(\ell)}>t-1.
\end{aligned}
\end{equation}
The four rows apply for $t\ge0$, $t\ge1$, $t\ge1$, and $t\ge2$,
respectively. The fixed $R[1]=I$ is the identity self-update and uses no
directed communication service. Device activation imposes the separate gates
\begin{equation}
\label{eq:hardware_gating}
\begin{aligned}
\eta_i=0&\Rightarrow M_{i,:}[t]=L_{i,:}[t]=0,\\
\xi_j=0&\Rightarrow N_{:,j}[t]=L_{:,j}[t]=0.
\end{aligned}
\end{equation}
The \(M\)- and \(N\)-gates hold for \(t\in\Z_{[1,\TFIR]}\), and the
\(L\)-gates hold for \(t\in\Z_{[0,\TFIR]}\).
Define the prescribed deployment-response class
\begin{equation}
\label{eq:implementation_local_class}
\mathcal F_{\TFIR}^{\mathrm{dep}}(d,\eta,\xi)
:=\{\Psi\in\mathcal F_{\TFIR}:\ \Psi\text{ satisfies }
\eqref{eq:implementation_mask_rule}\text{ and }\eqref{eq:hardware_gating}\}.
\end{equation}

Proposition~\ref{prop:deployment} provides an implementation certificate: the
response constraints are realizable using the deployment charged in
\(J_{\mathrm{arch}}\).

\begin{proposition}[Deployment sufficiency]
\label{prop:deployment}
Fix the site-assignment maps \(p_u,p_y,p_\beta\), and assume each selected
directed communication service satisfies its delivery bound \(d_{rs}\) without
loss or queueing. Then every
\(\Psi\in\mathcal F_{\TFIR}^{\mathrm{dep}}(d,\eta,\xi)\) admits a causal
nodewise implementation of \eqref{eq:ofsls_realization} using only the
enabled devices and selected communication services. It realizes
\(K=L-MR^{-1}N\) and the closed-loop responses \(\Psi\).
\end{proposition}
\begin{proof}
A sample generated at \(\tau\) arrives by \(\tau+d_{rs}\). Thus \(d_{rs}\le q\)
makes it available to an input-update term with \(\tau=t-q\) at time \(t\), and
\(d_{rs}\le q+1\) makes it available to a realization-state update at \(t+1\).
Substitution of \eqref{eq:realization_shifts} gives
\eqref{eq:implementation_mask_rule}; induction on $t$ makes each required sample available
when consumed. Equation~\eqref{eq:hardware_gating} removes disabled-device
signals. Algebraic elimination of \(\beta\) and the OF-SLS achievability
identities give the stated controller, internal stability, and responses.
\end{proof}
Other distributed realizations induce corresponding communication constraints.

\subsection{Control Architecture Co-Design Problem via SLS}

For a budget \(B\ge0\), define the primary hard-budget problem.
\begin{problem}[FIR-SLS architecture co-design]
\label{prob:implementation_local_sls}
\begin{equation}
\label{eq:implementation_local_sls}
\begin{aligned}
V_B:=\inf_{d,\eta,\xi,\Psi}\quad
&J_{\mathrm{FIR}}(\Psi)\\
\st\quad
&(d,\eta,\xi)\in\Aadm,\
\Psi\in\mathcal F_{\TFIR}^{\mathrm{dep}}(d,\eta,\xi),\\
&J_{\mathrm{arch}}(d,\eta,\xi)\le B.
\end{aligned}
\end{equation}
\end{problem}
The scalarized alternative minimizes
\(J_{\mathrm{FIR}}+\lambda J_{\mathrm{arch}}\), \(\lambda\ge0\), over the same
deployment class.

\section{Binary-Cumulative MICP Reformulation}
\label{sec:binary_sls_reformulation}

\subsection{Finite-Dimensional Encoding of Communication Delays}
The one-hot/cumulative map is an elementary encoding of a finite service
choice. Its role is to expose lag-dependent availability to the four deployment
gates; Proposition~\ref{prop:deployment} supplies the implementation certificate.
For each ordered site pair, one-hot variables select an option in
\(\mathcal D_{rs}\), and cumulative binaries encode its lag-dependent
availability:
\begin{definition}[Binary cumulative service encoding]
\label{def:binary_encoding}
\begin{equation}
\label{eq:e_increment}
\begin{aligned}
&e_{rs,\nu}\in\{0,1\},\quad \nu\in\mathcal D_{rs},
\sum_{\nu\in\mathcal D_{rs}}e_{rs,\nu}=1,\\
&s_{rs,q}\in\{0,1\},\quad
s_{rs,q}=
\sum_{\substack{\nu\in\mathcal D_{rs}\setminus\{\infty\}\\ \nu\le q}}
e_{rs,\nu},
\quad q\in\Z_{[0,\Tlat]},\\
&s_{rs,q}\le s_{rs,q+1},
\quad q\in\Z_{[0,\Tlat-1]}.
\end{aligned}
\end{equation}
Thus \(s_{rs,q}=1\) exactly when the selected service is available by lag
\(q\), and every cumulative sequence is binary and monotone.
\end{definition}

For a selected option \(d_{rs}\), \eqref{eq:e_increment} gives
\(e_{rs,\nu}=\mathbf 1_{\{d_{rs}=\nu\}}\) and
\(s_{rs,q}=\mathbf 1_{\{d_{rs}\le q\}}\). Conversely, every \((e,s)\)
satisfying \eqref{eq:e_increment} uniquely recovers \(d_{rs}\) as the option
\(\nu\) with \(e_{rs,\nu}=1\). The resulting bijection and its inverse are
\[
\begin{aligned}
\mathcal T: (d,\eta,\xi,\Psi)&\mapsto(e(d),s(d),\eta,\xi,\Psi),\\
\mathcal T^{-1}: (e,s,\eta,\xi,\Psi)&\mapsto(d(e),\eta,\xi,\Psi),
\end{aligned}
\]
where \(d_{rs}(e)=\nu\Longleftrightarrow e_{rs,\nu}=1\). Equivalently,
\(d_{rs}=\min\{q:s_{rs,q}=1\}\), with \(d_{rs}=\infty\) for an all-zero
cumulative sequence.

For nonconsecutive service delays, \(s_{rs,q}\) remains constant between
available options; the one-hot equality restricts it to \(\mathcal D_{rs}\).

The one-hot variables give the complete binary architecture cost
\begin{equation}
\label{eq:arch_cost_finite}
\begin{aligned}
J_{\mathrm{arch}}^{\mathrm{bin}}(e,\eta,\xi)
=\sum_i c_i^{\mathrm a}\eta_i
+\sum_j c_j^{\mathrm s}\xi_j
+\sum_{r,s\in\mathcal V}
\sum_{\nu\in\mathcal D_{rs}}
c_{rs}^{\mathrm c}(\nu)e_{rs,\nu}.
\end{aligned}
\end{equation}
Here \(c_{rs}^{\mathrm c}(\infty)=0\), so \eqref{eq:arch_cost_finite} preserves
\(J_{\mathrm{arch}}\). Device activation is represented exclusively by
\((\eta,\xi)\).

\begin{definition}[Encoded admissible image]
\label{def:badm_truncated}
\[
\BadmT
:=
\{(e(d),s(d),\eta,\xi):\ (d,\eta,\xi)\in\Aadm\}.
\]
Thus \(\BadmT\) preserves the service-option membership, device, compatibility,
cardinality, and mandatory or forbidden-resource restrictions in \(\Aadm\); it
is mixed-integer linear whenever those restrictions are.
\end{definition}

The maximum communication delay \(\Tlat\) and response horizon \(\TFIR\) are
independent. Cumulative availability is evaluated at \(\Tlat\) when gating
later response coefficients.

\subsection{FIR-SLS MICP}

For each integer \(q\ge0\), write \(\bar q:=\min\{q,\Tlat\}\). With the
cumulative variables, the deadlines in \eqref{eq:implementation_mask_rule} are
equivalent to
\begin{equation}
\label{eq:loc_indicator_coupling}
\begin{aligned}
s_{p_u(i),p_y(j),\bar t}=0
&\Longrightarrow L_{ij}[t]=0,\\
s_{p_u(i),p_\beta(k),\overline{t-1}}=0
&\Longrightarrow M_{ik}[t]=0,\\
s_{p_\beta(k),p_y(j),\bar t}=0
&\Longrightarrow N_{kj}[t]=0,\\
s_{p_\beta(k),p_\beta(\ell),\overline{t-1}}=0
&\Longrightarrow R_{k\ell}[t]=0.
\end{aligned}
\end{equation}
The four rows apply for \(t\ge0\), \(t\ge1\), \(t\ge1\), and \(t\ge2\),
respectively, over the FIR horizon. The fixed \(R[1]=I\) is the identity
self-update and uses no directed communication service. These service gates are
imposed together with the independent hardware gates \eqref{eq:hardware_gating};
service availability and device activation remain separate decisions.

\begin{problem}[FIR-SLS hard-budget MICP]
\label{prob:finite_micp}
\begin{equation}
\label{eq:finite_micp}
\begin{aligned}
V_B^{\mathrm{MI}}
:=\inf_{\Psi,e,s,\eta,\xi}
&J_{\mathrm{FIR}}(\Psi)\\
\st
&\Psi\in\mathcal F_{\TFIR},\quad
(e,s,\eta,\xi)\in\BadmT,\\
&\eqref{eq:loc_indicator_coupling}\ \text{and}\
\eqref{eq:hardware_gating}\ \text{hold},
J_{\mathrm{arch}}^{\mathrm{bin}}(e,\eta,\xi)\le B.
\end{aligned}
\end{equation}
\end{problem}
Fixed binaries leave a convex FIR OF-SLS subproblem; \eqref{eq:J_FIR} is
quadratic. The scalarized alternative minimizes
\(J_{\mathrm{FIR}}+\lambda J_{\mathrm{arch}}^{\mathrm{bin}}\), \(\lambda\ge0\),
under the same constraints.

\begin{theorem}[Exact finite reformulation]
\label{thm:finite_exact}
Let \(J_{\mathrm{FIR}}\) be proper, closed, and convex on
\(\mathcal F_{\TFIR}\). Fix \(\mathcal V\), the site-assignment maps, \(\Tlat\),
\(\TFIR\), the finite service-option sets
\(\{\mathcal D_{rs}\}_{r,s\in\mathcal V}\), the admissible set \(\Aadm\), and
\(B\), and suppose \(\BadmT\) has an exact mixed-integer linear
description. Then \emph{(i)} Problem~\ref{prob:finite_micp} is an MICP;
\emph{(ii)} \(\mathcal T\) bijects the feasible sets of
Problems~\ref{prob:implementation_local_sls} and~\ref{prob:finite_micp};
\emph{(iii)} \(\mathcal T\) leaves \(\Psi\) and
\(J_{\mathrm{FIR}}(\Psi)\) unchanged and satisfies
\(J_{\mathrm{arch}}^{\mathrm{bin}}(e(d),\eta,\xi)
=J_{\mathrm{arch}}(d,\eta,\xi)\); and
\emph{(iv)} \(V_B^{\mathrm{MI}}=V_B\), and for every \(\lambda\ge0\) the
scalarized problems have the same infimum. Whenever a common infimum is
attained, \(\mathcal T\) maps the corresponding optimizer sets bijectively.
\end{theorem}

\begin{proof}
Fixed binary decisions leave a convex problem in the FIR OF-SLS coefficients,
so the exact mixed-integer linear description of \(\BadmT\) proves
\emph{(i)}. The map \(\mathcal T\) defined after
Definition~\ref{def:binary_encoding} bijects \(\Aadm\) and \(\BadmT\). Since
\(s_{rs,q}=1\Longleftrightarrow d_{rs}\le q\), the four indicator rows in
\eqref{eq:loc_indicator_coupling} are coefficientwise equivalent to
\eqref{eq:implementation_mask_rule}, while the hardware gates are unchanged;
this proves \emph{(ii)}. The map leaves \(\Psi\) unchanged, and
\eqref{eq:arch_cost_finite} gives \emph{(iii)}. These identities preserve the
budget constraint and scalarized objectives, giving the value and optimizer-set
claims in \emph{(iv)}.
\end{proof}

Theorem~\ref{thm:finite_exact} records the exact finite problem equivalence;
Proposition~\ref{prop:deployment} links the response constraints to
implementation with the selected devices and services. A zero
mixed-integer gap certifies the optimum for the fixed sites, assignments,
horizons, finite option sets, admissible set, and objective.

\subsection{QI Controller-Support Reference}
\label{subsec:qi_reference}
QI classifies controller-transfer support, while the four realization-message
families define the deployed implementation. Each fixed deployment induces the
controller delay
\begin{equation}
\label{eq:controller_delay_induced}
\delta^K_{ij}(d,\eta,\xi):=
\begin{cases}
d_{p_u(i),p_y(j)}, & \eta_i\xi_j=1,\\
\infty, & \eta_i\xi_j=0,
\end{cases}
\end{equation}
and the controller subspace
\begin{equation}
\label{eq:SkD}
\mathcal S_K(d,\eta,\xi)
:=\{K\in\Rp^{m\times p}:K_{ij}[t]=0\ \text{for }
0\le t<\delta^K_{ij}\}.
\end{equation}
Thus a disabled actuator or sensor gives a structurally zero controller row or
column, represented by \(\delta^K_{ij}=\infty\).

The subspace \(\mathcal S\subseteq\Rp^{m\times p}\) is quadratically
invariant under \(P_{22}\) if \(KP_{22}K\in\mathcal S\) for all
\(K\in\mathcal S\) \cite{rotkowitz2005}.
Define the plant propagation latency matrix
\(\Pi=[\pi_{r\ell}]\in(\Z_{\ge1}\cup\{\infty\})^{n_y\times n_u}\) by
\begin{equation}
\label{eq:Dp_entry}
\pi_{r\ell}:=
\inf\{t\in\Z_{\ge0}:\ (P_{22})_{r\ell}[t]\neq0\},
\end{equation}
where \((P_{22})_{r\ell}[t]\in\R^{p_r\times m_\ell}\) is the corresponding
sensor--actuator block coefficient. With \(\inf\emptyset:=\infty\), finite
\(\pi_{r\ell}\) are at least one. The delay-domain criterion of
\cite[Theorem~2]{rotkowitzSimpleConditionConvexity2005} gives QI if and only if
\begin{equation}
\label{eq:Dqi}
\delta^K_{ir}+\pi_{r\ell}+\delta^K_{\ell j}\ge\delta^K_{ij},
\quad
\forall\, i,\ell\in\Z_{[1,n_u]},\
\forall\, j,r\in\Z_{[1,n_y]}.
\end{equation}
The left-hand side is the delay of the indirect path
\(y_j\to u_\ell\to y_r\to u_i\); QI requires this indirect path not to deliver
information from \(y_j\) to \(u_i\) earlier than the direct controller-support
delay \(\delta^K_{ij}\). It therefore classifies the controller-delay pattern
induced by each fixed deployment. Under \eqref{eq:Dqi}, every achievable
\(\Psi\) and its associated internally stabilizing \(K=L-MR^{-1}N\) satisfy
\(K\in\mathcal S_K(d,\eta,\xi)\Longleftrightarrow
L\in\mathcal S_K(d,\eta,\xi)\).

For a budget \(B\), define the budget-feasible QI comparison family
\begin{equation}
\label{eq:qi_budget_family}
\mathcal Q_B:=
\{(d,\eta,\xi)\in\Aadm:
J_{\mathrm{arch}}(d,\eta,\xi)\le B,
\ \text{\eqref{eq:Dqi} holds}\}.
\end{equation}
Each member is re-synthesized by the same fixed-architecture problem:
\begin{equation}
\label{eq:qi_uniform_resynthesis}
V_{\mathrm{QI}}(d,\eta,\xi)
:=\inf_{\Psi\in\mathcal F_{\TFIR}^{\mathrm{dep}}(d,\eta,\xi)}
J_{\mathrm{FIR}}(\Psi),
\end{equation}
where \(\inf\emptyset:=+\infty\). All QI deployments use the same candidates,
assignments, option sets, costs, response gates, and hardware gates. The best
budget-feasible QI reference minimizes \eqref{eq:qi_uniform_resynthesis} over
\(\mathcal Q_B\); QI enforces \eqref{eq:SkD} through affine support on \(L\).

\section{Numerical Validation}
\label{sec:numerics}

\begin{figure}[t]
\centering
\includegraphics[width=\columnwidth]{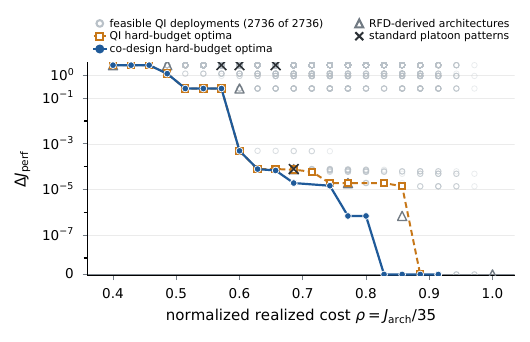}
\caption{Platoon architecture co-design at \(\TFIR=10\).}
\label{fig:platoon_budget_comparison}
\end{figure}

We use the homogeneous three-follower constant-time-headway platoon of
\cite{zhengStabilityScalabilityHomogeneous2016} to evaluate joint actuator,
sensor-package, and communication-service selection under a hard architecture
budget. The headway is \(h=0.6\,\mathrm{s}\), the actuator lag is
\(\tau=0.25\,\mathrm{s}\), and the sampling time is \(0.1\,\mathrm{s}\). For
\(i\in\Z_{[1,3]}\),
\(x_i=[e_i\ q_i\ \alpha_i]^\top\) collects spacing error, relative velocity,
and acceleration; \(\alpha_0\) is the leader acceleration, and \(\omega_i\) is
an actuation disturbance. The dynamics are
\(\dot e_i=q_i-h\alpha_i\),
\(\dot q_i=\alpha_{i-1}-\alpha_i\), and
\(\tau\dot\alpha_i=-\alpha_i+u_i+\omega_i\). Zero-order-hold discretization gives
\(n_x=9\), \(m=n_u=3\), \(n_y=7\), \(p=10\), \(n_w=4\), and
\(n_\zeta=11\), with
\begin{equation}
\label{eq:platoon_signals}
w=\col(\alpha_0,\omega_1,\omega_2,\omega_3),\qquad
y=\col(y_1,\ldots,y_7).
\end{equation}
The measurement blocks are \(y_1=\alpha_0\),
\(y_{2i}=\col(e_i,q_i)\), and \(y_{2i+1}=\alpha_i\) for
\(i\in\Z_{[1,3]}\). Thus \(n_y=7\) counts physical sensor packages, whereas \(p=10\) counts their
scalar outputs, with \((p_1,\ldots,p_7)=(1,2,1,2,1,2,1)\). Here
\(C_2=\col(0_{1\times9},I_9)\), \(D_{21}(1,1)=1\), and all other entries of
\(D_{21}\) are zero; no measurement-noise input is added. Define the selection
matrices \(E_e,E_q,E_\alpha\) by
\(E_ex=\col(e_1,e_2,e_3)\),
\(E_qx=\col(q_2-q_1,q_3-q_2)\), and
\(E_\alpha x=\col(\alpha_1,\alpha_2,\alpha_3)\). The regulated output is
\begin{equation}
\label{eq:platoon_performance}
\zeta=\col(\sqrt{10}E_ex,E_qx,\sqrt{0.5}E_\alpha x,\sqrt{0.1}u),
\end{equation}
which specifies \(C_1\) and \(D_{12}\); \(D_{11}=0\).

Sites are \(\{0,1,2,3\}\), with the leader at 0 and follower \(i\) at \(i\).
Three realization-state blocks use \(p_u(i)=p_\beta(i)=i\) and
\((p_y(1),\ldots,p_y(7))=(0,1,1,2,2,3,3)\). Following CACC instrumentation
\cite{deHaan2024Platooning}, the seven sensor packages are the leader
acceleration and, at each follower, one ranging/range-rate and one onboard
acceleration package; \(\xi\) uses this order. Every device has normalized unit
cost. Local services have zero delay and cost; nonlocal delay-one and delay-two
options cost three and two units, so faster delivery has higher cost. Let
\(d^{\mathrm F}:=[d_{rs}]_{r=1,2,3;\,s=0,1,2,3}\); its rows and columns
index follower destination sites and source sites, respectively. The three
leader and four directed adjacent-
follower services choose from \(\{1,2,\infty\}\); the two distance-two
services choose from \(\{2,\infty\}\). Thus there are \(3^7 2^2=8748\) admissible
service matrices and \(8748\,2^{10}=8{,}957{,}952\) joint device--service
deployments. All ten hardware binaries are free in the main hard-budget solves.

PBH tests find that only \(\eta=\mathbf 1_3\) is stabilizable and 16 of the 128
sensor subsets are detectable. The three ranging/range-rate packages are necessary; the
leader and follower acceleration packages are optional.

We solve Problem~\ref{prob:finite_micp} for every integer
\(B\in\{14,\ldots,35\}\) with \(\TFIR=10\), using Gurobi~13.0.1, Seed 23,
Threads 1, NumericFocus 3, and FeasibilityTol, OptimalityTol, IntFeasTol, and
MIPGap set to \(10^{-9}\). All device and service variables are
optimized under \(J_{\mathrm{arch}}\le B\). The dense minimum-delay deployment
has \(J_{\mathrm{arch}}^{\mathrm{dense}}=35\) and
\(J_{\mathrm{dense}}=4.9874267573\). For each evaluated deployment--response
pair, define \(\rho:=J_{\mathrm{arch}}(d,\eta,\xi)/
J_{\mathrm{arch}}^{\mathrm{dense}}\) and
\(\Delta J_{\mathrm{perf}}:=J_{\mathrm{FIR}}(\Psi)-J_{\mathrm{dense}}\).
Every reported co-design and QI budget optimum has zero mixed-integer gap and
passes an independent fixed-architecture OF-SLS residual check at
\(2\times10^{-8}\).

We compare with the QI family of Section~III-C, RFD-derived architectures, and
six standard platoon information-flow patterns from
\cite{zhengStabilityScalabilityHomogeneous2016}: predecessor following (PF),
predecessor--leader following (PLF), bidirectional (BD), bidirectional--leader
(BDL), two-predecessor following (TPF), and two-predecessor--leader following
(TPLF). Each standard pattern installs its prescribed site pairs at their
fastest available options, omits the remaining optional nonlocal pairs, and
reselects hardware under the same budget. RFD groups the FIR coefficients into actuator rows of \((M,L)\),
sensor-package columns of \((N,L)\), and optional directed services. Nested
service groups use incremental option-cost weights for successively faster
deadlines. We scan
\(\lambda\in\{0\}\cup\{a\,10^k:a\in\{1,3\},k\in\Z_{[-4,2]}\}\cup\{10^3\}\)
and thresholds \(\{10^k:k=-8,\ldots,0\}\cup\{5,10\}\). We map all 176 points to
their actual discrete cost, then re-synthesize each of the 12 unique discrete
deployments with its architecture fixed.

Figure~\ref{fig:platoon_budget_comparison} compares all classes under the same
FIR horizon, device candidates, service menus, and response gates. Each budget
optimum satisfies \(J_{\mathrm{arch}}\le B\), but is plotted at realized cost
\(\rho\); optima for the same budget may therefore have different abscissae.

At \(B=24\), hard-budget co-design selects \(\eta=\mathbf 1_3\),
\(\xi=(1,1,0,1,0,1,0)^\top\), and a cost split of \(3+4+17=24\) between
actuators, sensors, and services. It gives
\(\Delta J_{\mathrm{perf}}=1.92256\times10^{-5}\). The best QI design selects
\(\eta=\mathbf 1_3\), \(\xi=(1,1,1,1,0,1,0)^\top\), costs \(3+5+16=24\), and
gives \(7.41246\times10^{-5}\), a factor of \(3.86\). This comparison is
exhaustive within the PBH-admissible finite catalog: it contains \(139{,}968\)
device--service deployments, of which \(2736\) are QI-compatible. All 2736
fixed-architecture QPs return optimal solutions that pass the independent
residual check; the corresponding deployments are shown in
Fig.~\ref{fig:platoon_budget_comparison}.
The co-design and QI follower-destination service matrices are
\[
{\setlength{\arraycolsep}{3pt}
\begin{aligned}
d^{\mathrm F,\mathrm{co}}_{24}
&=\begin{bmatrix}
\infty&0&1&\infty\\
1&1&0&\infty\\
1&2&1&0
\end{bmatrix},
&
d^{\mathrm F,\mathrm{QI}}_{24}
&=\begin{bmatrix}
\infty&0&\infty&\infty\\
1&1&0&2\\
1&2&1&0
\end{bmatrix}.
\end{aligned}}
\]
With the sensor-block order in \eqref{eq:platoon_signals}, taking
\((i,r,\ell,j)=(1,4,2,1)\) gives
\(\delta^K_{1,4}+\pi_{4,2}+\delta^K_{2,1}=1+1+1
<\delta^K_{1,1}=\infty\). Hence the co-design optimum is non-QI by
\eqref{eq:Dqi}.

All 12 RFD-induced fixed-architecture QPs pass the residual check; seven are
non-dominated. RFD recovers the co-design deployment at \(B=14\); at \(B=24\),
its best loss, \(7.87080\times10^{-5}\), equals the best standard-pattern
reference (TPLF) but exceeds the co-design optimum.

Fixing all devices while retaining selectable services is infeasible at
\(B=14\); at \(B=24\), it gives
\(\Delta J_{\mathrm{perf}}=7.87080\times10^{-5}\), versus
\(1.92256\times10^{-5}\) with selectable hardware, and at \(B=35\) it reaches
\(J_{\mathrm{dense}}\) at cost 34. Thus sensor selection affects feasibility
and performance. Adding a delay-three nonlocal option of cost one leaves the
co-design optima unchanged at \(B=14,24,35\); at \(B=24\), the QI loss falls
from \(7.41246\times10^{-5}\) to \(6.82045\times10^{-5}\), still above the
co-design value. For \(\TFIR=8,12\) and \(B=14,24,35\), respectively,
\(J_{\mathrm{FIR}}\) changes from \((26.3237,11.71715,11.71714)\) to
\((3.46280,2.72241,2.72235)\), while the selected hardware and service matrices
remain unchanged over \(\TFIR\in\{8,10,12\}\). The code and numerical data are
archived in the versioned public release
\cite{zhouCoDeCALCSSReproduction2026}.

\section{Discussion and Conclusion}
\label{sec:discussion}

For discrete-time LTI plants, this paper presents a co-design framework that
jointly selects actuators, sensor packages, and communication services while
synthesizing OF-SLS responses under a hard architecture budget. Its main
technical result is an exact finite MICP reformulation for the prescribed
deployment class: cumulative binaries encode finite service choices, and the
response and device gates ensure that every feasible response admits a causal
implementation with the selected resources. At \(B=24\), the platoon optimum
retains the PBH-required actuators, removes optional sensors, and outperforms
the budget-feasible QI, standard-pattern, and RFD-derived alternatives.

Scaling the co-design to larger networks leads to a structured screening
problem: which candidate devices and communication services can be removed
without materially degrading closed-loop performance? Spatial decay in
networked control \cite{shinNearOptimalDistributed2023} motivates
graph-neighborhood and response-based screening. Reduced MICPs can then be
benchmarked against the full formulation on tractable instances to quantify
performance loss and computational savings. Future work will develop certified
screening conditions and hierarchical architecture search for larger networks.

\bibliographystyle{IEEEtran}
\bibliography{CoDeCA_LCSS_V2}

\end{document}